\documentclass[1p]{elsarticle}

\usepackage{lineno}
\modulolinenumbers[5]

\journal{ }

\usepackage{mathrsfs}
\usepackage{amsfonts}
\usepackage[notref,notcite]{showkeys}
\usepackage{enumerate}
\usepackage{latexsym}
\usepackage[all,cmtip]{xy}
\usepackage{verbatim}

\usepackage[colorlinks,linkcolor=blue,citecolor=blue]{hyperref}

\usepackage{amsmath,amssymb,xspace,amsthm}

\numberwithin{equation}{section}

\def\gl{\mathfrak{gl}}

\def\fa{\mathfrak{a}}
\def\fb{\mathfrak{b}}
\def\fg{\mathfrak{g}}
\def\fh{\mathfrak{h}}
\def\fri{\mathfrak{i}}
\def\fm{\mathfrak{m}}
\def\fD{\mathfrak{D}}

\def\cK{\mathcal{K}}

\def\cU{\mathcal{U}}
\def\cI{\mathcal{I}}
\def\cT{\mathcal{T}}

\def\bC{\mathbb{C}}
\def\bN{\mathbb{N}}
\def\bZ{\mathbb{Z}}

\def\supp{\mathrm{supp}}

\newcommand\ol{\overline}
\newcommand\wt{\widetilde}

\newtheorem{theo}{{Theorem}}[section]
\newtheorem{lemm}[theo]{Lemma}
\newtheorem{rema}[theo]{Remark}

\newtheorem{prop}[theo]{Proposition}

\allowdisplaybreaks

\begin{document}

\begin{frontmatter}



\title{Simple Harish-Chandra modules over the super affine-Virasoro algebras}


\author[1]{Yan He}
\ead{heyan913533012@163.com}

\author[2]{Dong Liu}
\ead{liudong@zjhu.edu.cn}

\author[3]{Yan Wang\corref{cor1}}
\ead{wangyan09@tju.edu.cn}

\address[1]{Department of Mathematics, Changshu Institute of Technology, Suzhou P. R. China}
\address[2]{Department of Mathematics, Huzhou University, Zhejiang P. R. China}
\address[3]{School of Mathematics, Tianjin University, Tianjin, P. R. China}

\cortext[cor1]{Corresponding author}

\begin{abstract}
In this paper, we classify all simple Harish-Chandra modules over the super affine-Virasoro algebra $\widehat{\mathcal{L}}=\mathcal{W}\ltimes(\fg\otimes \mathcal{A})\oplus \mathbb{C}C$, where $\mathcal{A}=\bC[t^{\pm 1}]\otimes\Lambda(1)$ is the tensor superalgebra of the Laurent polynomial algebra in even variable $t$ and the Grassmann algebra in odd variable $\xi$, $\mathcal{W}$ is the Lie superalgebra of superderivations of $\mathcal{A}$, and $\fg$ is a finite-dimensional perfect Lie superalgebra.
\end{abstract}


\end{frontmatter}

\section{Introduction}
Throughout this paper, we denote by $\bZ, \bZ_+, \bN$ and $\bC$ the sets of all integers, non-negative integers,
positive integers and complex numbers, respectively. All vector spaces and algebras in this paper are over $\bC$.
A super vector space $V$ is a vector space endowed with a $\mathbb{Z}_2$-gradation $V=V_{\bar{0}}\oplus V_{\bar{1}}$.
The parity of a homogeneous element $v\in V_{\bar{i}}$ is denoted by $|v|=\bar{i}\in \mathbb{Z}_2$. When we write $|v|$
for an element $v\in V$, we will always assume that $v$ is a homogeneous element. We denote by $U(L)$ the universal
enveloping algebra of the Lie (super)algebra $L$. Also, we denote by $\delta_{i,j}$ the Kronecker delta.

Let $\mathcal{A}=\bC[t^{\pm 1}]\otimes\Lambda(1)$ be the tensor superalgebra of the Laurent polynomial algebra in even variable $t$ and the Grassmann algebra in odd variable $\xi$,  and $\fg=\fg_{\bar{0}}\oplus \fg_{\bar{1}}$ be a finite-dimensional perfect Lie superalgebra (i.e. $\fg=[\fg, \fg]$). Then $\fg \otimes \mathcal{A}$ is a Lie superalgebra, which is called the super loop algebra (see \cite{KT}, also named current Lie superlagebra in \cite{NY}), with $[x\otimes a, y\otimes b]=(-1)^{|a||y|}[x, y]\otimes ab$ for all $x, y\in \fg, a, b\in \mathcal{A}$. The universal central extensions of $\fg\otimes \mathcal{A}$ equipped with a nondegenerated homogeneous invariant supersymmetric bilinear form on $\fg$ were studied in \cite{NY, NY1} recently.

Let $\mathcal{W}$ be the Witt superalgebra, i.e. the Lie superalgebra of superderivations of $\mathcal{A}$. Clearly $\fg \otimes \mathcal{A}$ becomes a $\mathcal{W}$-module by the usual actions of $\mathcal{W}$ on $\mathcal{A}$. So we can define a Lie
superalgebra $\mathcal{L}$ associated to $\fg$ as $\mathcal{L}=\mathcal{W}\ltimes(\fg\otimes \mathcal{A})$.
 By calculating the second cohomology group $\mbox{H}^2(\mathcal{L}, \bC)$ of $\mathcal{L}$,
we obtain the universal central extension $\widehat{\mathcal{L}}=\mathcal{W}\ltimes(\fg\otimes \mathcal{A})\oplus \mathbb{C}C$, which is called the super affine-Virasoro algebra, also named  super conformal current algebra in \cite{KT}. It
can be viewed as a super version of the affine-Virasoro algebra defined in \cite{K} (see also \cite{LPX1}) and it corresponds to a superconformal and chiral invariant 2-dimensional quantum field theory (see \cite{KT}).

The affine-Virasoro algebra is the semi-direct sum of the Virasoro algebra and the untwisted affine Lie algebra. There have been many researches on the representation theory of the affine-Virasoro algebras, see references \cite{CY, EJ, GHL, LPX1}, and so on. However, the research concerning the representation theory of the super affine-Virasoro algebra is still seldom. All unitary irreducible representations for the subalgebra $\widehat{\frak g}_R=R\ltimes(\dot\fg\otimes \mathcal{A})\oplus \mathbb{C}C$ of $\widehat{\mathcal{L}}$ were constructed in \cite{KT}, where $\dot\fg$ is a semisimple Lie algebra and $R$ is the centerless $N=1$ superconformal algebra, a subalgebra of $\mathcal{W}$.

Based on the classification of simple jet modules introduced by Y.Billig in \cite{B} (see also \cite{E}), a complete classification of simple Harish-Chandra modules over
Lie algebra of vector fields on a torus was given in \cite{BF} by  the so-called $\mathcal{A}$-cover theory. As we know, the classification of cuspidal modules is one of
the most important steps in the classification of simple Harish-Chandra modules over various Lie (super)algebras.
Using the $\mathcal{A}$-cover theory, the classifications of simple Harish-Chandra modules over the Witt superalgebra (\cite{XL2})(also see \cite{BFI}), the $N=1$ superconformal algebra(\cite{CL, CLL}), the map (super)algebra related to the Virasoro algebra (\cite{CLW}) were given. Certainly such researches for the $N=1, 2$ superconformal algebras, the affine-Virasoro algebra were first given in \cite{LPX1, LPX2, S} by other methods, respectively. Motivated by the above researches, we classify the simple Harish-Chandra modules over the super affine-Virasoro algebra $\widehat{\mathcal{L}}$ in this paper.

The paper is organized as follows. In Section 2, we give some definitions and preliminaries. In Section 3, we study the central extension of $\mathcal{L}$ and get the super affine-Virasoro algebra $\widehat{\mathcal{L}}$. The $\mathcal{A}\mathcal{L}$-modules are studied in Section 4. In Section 5, using the $\mathcal{A}$-cover theory and the results of $\mathcal{A}\mathcal{L}$-modules, we give the classification of
simple cuspidal modules of $\mathcal{L}$. Finally, we prove our main theorem in Section 6, see Theorem \ref{Theorem 3}.

\section{Preliminaries}
In this section, we recall some necessary definitions and preliminary results.

\subsection{The super affine-Virasoro algebra}

Let $\mathcal{A}=\bC[t^{\pm 1}]\otimes\Lambda(1)$ be the tensor superalgebra of the Laurent polynomial algebra in even variable $t$ and the Grassmann algebra in odd variable $\xi$,
and $\mathcal W$ be the Witt superalgebra. Denote by $d_i=t^{i+1}\frac{\partial}{\partial t}$, $h_i=t^{i}\xi\frac{\partial}{\partial\xi}$, $Q_i=t^{i}\frac{\partial}{\partial\xi}$, $G_i=t^{i+1}\xi\frac{\partial}{\partial t}$ for any $i\in \mathbb{Z}$, and $\Delta={\rm span}_\bC\{t\frac{\partial}{\partial t}, \frac{\partial}{\partial\xi}\}$. Then $\mathcal{W}={\mathcal A}\Delta={\rm span}_\bC \{d_i, h_i, Q_i, G_i\mid i\in \mathbb{Z}\}$ with brackets given by
\begin{align*}
&[d_i,d_j]=(j-i)d_{i+j},\;[d_i,h_j]=jh_{i+j}, \;[d_i,Q_j]=jQ_{i+j}, \;[d_i,G_j]=(j-i)G_{i+j},\\
&[h_i,Q_j]=-Q_{i+j}, \;[h_i,G_j]=G_{i+j},\;[Q_i,G_j]=d_{i+j}+ih_{i+j}.
\end{align*}
Obviously, $\mathfrak W={\rm span}_\bC\{d_i\mid i\in \mathbb{Z}\}$ is the Witt algebra. It is well known that $\mathcal{W}$ is isomorphic to the $N=2$ (centerless) Ramond algebra (see \cite{Kv, LPX2}).

Let $\fg=\fg_{\bar{0}}\oplus \fg_{\bar{1}}$ be a finite-dimensional perfect Lie superalgebra (i.e. $\fg=[\fg, \fg]$).
Then $\fg \otimes \mathcal{A}$ becomes a $\mathcal{W}$-module (resp. $\mathfrak{W}$-module) by the usual actions of $\mathcal{W}$ (resp. $\mathfrak{W}$) on $\mathcal{A}$.
So we can define a Lie superalgebra $\mathcal{L}$ (resp. $\mathfrak{L}$) associated to $\fg$ as $\mathcal{L}=\mathcal{W}\ltimes(\fg\otimes \mathcal{A})$ (resp. $\mathfrak{L}=\mathfrak{W}\ltimes(\fg\otimes \mathcal{A})$). It is easy to see that $\mathfrak{L}$ is a super subalgebra of $\mathcal{L}$. In addition to the brackets on $\mathcal{W}$, the rest of the nonzero brackets in $\mathcal{L}$ are as follows:
\begin{align*}
&[d_i,x\otimes t^j\xi]=jx\otimes t^{i+j}\xi,\ [d_i,x\otimes t^j]=jx\otimes t^{i+j},\\
&[h_i,x\otimes t^j\xi]=x\otimes t^{i+j}\xi,\ [Q_i,x\otimes t^j\xi]=(-1)^{|x|}x\otimes t^{i+j},\\
&[G_i,x\otimes t^j]=(-1)^{|x|}jx\otimes t^{i+j}\xi,\ [x\otimes t^i, y\otimes t^j]=[x,y]\otimes t^{i+j},\\
&[x\otimes t^i\xi, y\otimes t^j]=(-1)^{|y|}[x,y]\otimes t^{i+j}\xi,
\end{align*}
where $i,j\in\bZ$ and $x, y\in\fg$.

The universal central extension $\widehat{\mathcal L}$ (resp.  $\widehat{\mathfrak L}$) of $\mathcal L$  (resp.  $\mathfrak L$) is called a super affine-Virasoro algebra (or super conformal current algebra in \cite{KT}).

Denoted by $\cK$ the associative superalgebra generated by $\mathcal{A}$ and $\Delta$, which is called the super Weyl algebra. For any $\lambda\in\bC$, let $\sigma_\lambda$ be the automorphism of $\cK$ with $\sigma_\lambda(d_i)=d_i+\lambda, \sigma_\lambda(\frac{\partial}{\partial \xi})=\frac{\partial}{\partial \xi}, \sigma_\lambda|_\mathcal{A}=\mbox{id}_\mathcal{A}$. Denote $\mathcal{A}(\lambda):=\mathcal{A}^{\sigma_\lambda}$. It is clear that $\mathcal{A}(\lambda)\cong \cK/\cI_\lambda$, where $\cI_\lambda$ is the left ideal of $\cK$ generated by $d_i-\lambda$ and $\frac{\partial}{\partial \xi}$. We need the following lemmas.
\begin{lemm}\label{XL2-2}(\cite{XL2}, Lemma 3.5)
1. $\mathcal{A}(\lambda)$ is a strictly simple $\cK$-module.

2. Any simple weight $\cK$-module is isomorphic to some $\mathcal{A}(\lambda)$ for some $\lambda\in \bC$ up to a parity-change.
\end{lemm}
\begin{lemm}\label{L1} For the Lie superalgebra $\mathcal L$, we have the following relations:
\begin{align*}
&[(t-1)^kd_i,(t-1)^l d_j]=(l-k+j-i)(t-1)^{k+l}d_{i+j}+(l-k)(t-1)^{k+l-1}d_{i+j},\\
&[(t-1)^kd_i,(t-1)^l h_j]=(l+j)(t-1)^{k+l}h_{i+j}+l(t-1)^{k+l-1}h_{i+j},\\
&[(t-1)^kd_i,(t-1)^l Q_j]=(l+j)(t-1)^{k+l}Q_{i+j}+l(t-1)^{k+l-1}Q_{i+j},\\
&[(t-1)^kd_i,(t-1)^l G_j]=(l-k+j-i)(t-1)^{k+l}G_{i+j}+(l-k)(t-1)^{k+l-1}G_{i+j},\\
&[(t-1)^kh_i,(t-1)^l Q_j]=-(t-1)^{k+l}Q_{i+j},\\
&[(t-1)^kh_i,(t-1)^l G_j]=(t-1)^{k+l}G_{i+j},\\
&[(t-1)^kQ_i,(t-1)^l G_j]=(t-1)^{k+l}d_{i+j}+(i+k)(t-1)^{k+l}h_{i+j}+k(t-1)^{k+l-1}h_{i+j},\\
&[(t-1)^kh_i,(t-1)^l h_j]=[(t-1)^kQ_i,(t-1)^l Q_j]=[(t-1)^kG_i,(t-1)^l G_j]=0,\\
&[(t-1)^kd_i,x\otimes(t-1)^l t^j]=jx\otimes(t-1)^{k+l}t^{i+j}+lx\otimes(t-1)^{k+l-1}t^{i+j+1},\\
&[(t-1)^kd_i,x\otimes(t-1)^l t^j\xi]=lx\otimes(t-1)^{k+l-1}t^{i+j+1}\xi,\\
&[(t-1)^kh_i,x\otimes(t-1)^l t^j\xi]=x\otimes(t-1)^{k+l}t^{i+j}\xi,\\
&[(t-1)^kQ_i,x\otimes(t-1)^l t^j\xi]=(-1)^{|x|}x\otimes(t-1)^{k+l}t^{i+j},\\
&[(t-1)^kG_i,x\otimes(t-1)^l t^j]=(-1)^{|x|}(jx\otimes(t-1)^{k+l}t^{i+j}\xi+lx\otimes(t-1)^{k+l-1}t^{i+j+1}\xi),\\
&[(t-1)^kh_i,x\otimes(t-1)^l t^j]=[(t-1)^kQ_i,x\otimes(t-1)^lt^j]=0=[(t-1)^kG_i,x\otimes(t-1)^l t^j\xi]
\end{align*} for all $k,l\in\bZ_+$ and $x\in\fg$.
\end{lemm}
\begin{proof}
The results follow from direct computations.
\end{proof}

\subsection{The central extension of Lie superalgebras}
A central extension $\widehat{L}$ of the Lie superalgebra $L$ is a short exact sequence of the Lie superalgebras
$$0\rightarrow\mathfrak{c}\xrightarrow[]{i}\widehat{L}\xrightarrow[]{\tau}L\rightarrow 0,$$
where $\mathfrak{c}$ is a commutative Lie algebra over $\bC$, i.e. $\mathfrak{c}_{\bar{1}}=0$ and $[\mathfrak{c}, \widehat{L}]=0$. Sometimes we denote the above central extension by a pair $(\widehat{L},\tau)$ and call a central extension of $L$ by $\mathfrak{c}$. A central extension $(\widehat{L},\tau)$ is called universal if for any central extension $(\widetilde{L}, \iota)$ of $L$ there exists a unique homomorphism $\psi: \widehat{L}\rightarrow \widetilde{L}$ such that $\iota\circ\psi =\tau$. A Lie superalgebra
$L$ has the universal central extension if and only if $L$ is perfect.

It is well known that $\widehat{L}=L\oplus \bC C$ is a $1$-dimensional central extension of Lie superalgebra $L$ if and only if $\widehat{L}$ is the
direct sum of $L$ and $\bC C$ as vector spaces and the bracket $[\cdot,\cdot]_1$ in $\widehat{L}$ is given by
$$[x,y]_1=[x,y]+\alpha(x,y)C,\ \ [x, C]_1=0$$
for all $x, y\in L$, where $[\cdot,\cdot]$ is the bracket in $L$ and $\alpha: L\times L\rightarrow \bC$ is a bilinear form on
$L$ satisfying the following conditions
\begin{align*}
&\alpha(x,y)=-(-1)^{|x||y|}\alpha(y,x),\\
&\alpha(x,[y,z])=\alpha([x,y],z)+(-1)^{|x||y|}\alpha(y,[x,z])
\end{align*}
for $x,y,z\in L$. The bilinear form $\alpha$ is called a 2-cocycle on $L$.  A 2-cocycle is called a 2-coboundary if there is a linear function
$\rho$ from $L$ to $\bC$ such that $\alpha(x,y)=\rho([x,y])$ for all $x,y\in L$. The set of all 2-cocycles on $L$ is a vector space, denoted by $Z^2(L, \bC)$.
The set of all 2-coboundaries is a subspace of $Z^2(L, \bC)$, denoted by $B^2(L, \bC)$.
From \cite{IK}, the set of equivalence classes of such central extensions are known to be parameterized by the second
cohomology group $\mbox{H}^2(L, \bC)=Z^2(L,\bC)/B^2(L, \bC)$.

\subsection{Weight modules}
For the Lie superalgerba $\mathcal L$ (resp. $\mathfrak L$) defined in Section 2.1,  an $\mathcal L$-module (resp. $\mathfrak L$-module) $V$ is called a weight module if the action of $d_0$ on $V$ is diagonalizable, i.e. $V=\bigoplus\limits_{\lambda\in \bC} V_{\lambda}$, where $V_{\lambda}=\{v\in V\mid d_0 v=\lambda v\}$.  The support set of a weight module $V$ is defined by $\supp(V)=\{\lambda\in \bC\mid V_{\lambda}\ne 0\}$. A weight $\mathcal L$-module (resp. weight $\mathfrak L$-module) $V$ is called Harish-Chandra if $\dim V_{\lambda}<\infty,\forall \lambda\in \supp(V)$, and is called
cuspidal or uniformly bounded if  there exists some $N\in \bN$ such that $\dim V_{\lambda}\le N,\forall \lambda\in \supp(V)$.

For the above $\mathcal{A}$ and $\fg$, set $\wt{\mathcal{L}}=\mathcal{W}\ltimes((\fg\otimes \mathcal{A})\oplus \mathcal{A})$. An $\wt{\mathcal{L}}$-module $V$ is called an $\mathcal{A}\mathcal{L}$-module if $\mathcal{A}$ acts associatively, i.e. $t^0 v=v, fgv=f(gv),\forall f,g\in \mathcal{A}, v\in V$. Let $V$ be a weight $\mathcal{A}\mathcal{L}$-module and $V=\bigoplus\limits_{\lambda\in \bC} V_{\lambda}$. For any $v\in V_\lambda$ and $i\in\bZ$, we have
\begin{align*}
&d_0d_i v=(\lambda+i)d_iv,\;\;\;d_0h_iv=(\lambda+i)h_iv,\\
&d_0Q_iv=(\lambda+i)Q_iv,\;\;\;d_0G_iv=(\lambda+i)G_iv,\\
&d_0(x\otimes t^i) v=(\lambda+i)x\otimes t^i v,\;\;\;d_0(x\otimes t^i\xi) v=(\lambda+i)x\otimes t^i\xi v,\\
&d_0 t^i v=(\lambda+i)t^i v,\;\;\;d_0 t^i\xi v=(\lambda+i)t^i\xi v.
\end{align*}
Thus if $V$ is simple, then $\supp(V)=\lambda+\bZ$ for some $\lambda\in\bC$.

\subsection{Some useful results}
A module $M$ over an associative superalgebra $B$ is called strictly simple if it is a simple module over the associative algebra $B$ (forgetting the $\bZ_2$-gradation).
\begin{lemm}\label{XL2-1}
(\cite{XL2}, Lemma 2.1, 2.2)Let $B, B'$ be unital associative superalgebras, and $M, M'$ be $B, B'$ modules, respectively.

1. $M\otimes M'\cong \Pi(M)\otimes \Pi(M'^T)$ as $B\otimes B'$-modules.

2. If $B'$ has a countable basis and $M'$ is strictly simple, then

(1) any $B\otimes B'$-submodule of $M\otimes M'$ is of the form $N\otimes M'$ for some $B$-submodule $N$ of $M$;

(2) any simple quotient of the $B\otimes B'$-module $M\otimes M'$ is isomorphic to some $\overline{M}\otimes M'$ for some simple quotient $\overline{M}$ of $M$;

(3) $M\otimes M'$ is a simple $B\otimes B'$-module if and only if $M$ is a simple $B$-module;

(4) if $V$ is a simple $B\otimes B'$-module containing a strictly simple $B'=\bC\otimes B'$-module $M'$, then $V\cong M\otimes M'$ for some simple $B$-module
$M$.
\end{lemm}

\section{The universal central extension of $\mathcal{L}$}
In this section, we discuss the structure of the universal central extension of $\mathcal{L}=\mathcal{W}\ltimes(\fg\otimes \mathcal{A})$ by the $1$-dimensional center, where $\fg$ is a finite-dimensional perfect Lie superalgebra.

Let $\alpha^\prime$ be a 2-cocycle on $\mathcal{L}$ and $\rho$ be a linear function from $\mathcal{L}$ to $\bC$.
Set $\alpha=\alpha^\prime+\alpha_{\rho}$, where $\alpha_{\rho}\in B^2(\mathcal{L}, \bC)$ and $\alpha_{\rho}(x, y)=\rho([x,y])$. Then $\alpha$ is a 2-cocycle on $\mathcal{L}$ which
is equivalent to $\alpha^\prime$. For a given $x\in\fg$, we define
\begin{align*}
&\rho(x\otimes 1)=\alpha^\prime(d_1, x\otimes t^{-1}),\ \ \rho(x\otimes t^k)=-\frac{1}{k}\alpha^\prime(d_0, x\otimes t^k), k\neq 0,\\
&\rho(x\otimes \xi)=\alpha^\prime(d_1, x\otimes t^{-1}\xi), \ \ \rho(x\otimes t^k\xi)=-\frac{1}{k}\alpha^\prime(d_0, x\otimes t^k\xi), k\neq 0.
\end{align*}
So $\alpha(d_0, x\otimes t^k)=\alpha(d_0, x\otimes t^k\xi)=0$ for $k\neq 0$.

By the definition of 2-cocycle, we have
$$\alpha(d_0, [d_i, x\otimes t^j])=\alpha([d_0, d_i], x\otimes t^j)+\alpha(d_i, [d_0, x\otimes t^j]).$$
Because $\alpha(d_0, x\otimes t^k)=0$, there is $(i+j)\alpha(d_i, x\otimes t^j)=0$. Then $\alpha(d_i, x\otimes t^j)=0, i+j\neq 0$.
Similarly, for $i+j\neq 0$, we get
\begin{align*}
\alpha(h_i, x\otimes t^j)&=\alpha(Q_i, x\otimes t^j)=\alpha(G_i, x\otimes t^j)=0,\\
\alpha(d_i, x\otimes t^j\xi)&=\alpha(h_i, x\otimes t^j\xi)=\alpha(Q_i, x\otimes t^j\xi)=\alpha(G_i, x\otimes t^j\xi)=0.
\end{align*}

Let $p_x(i)=\alpha(d_i, x\otimes t^{-i})$. Since
$$\alpha(d_{i+j}, [d_{-i}, x\otimes t^{-j}])=\alpha([d_{i+j}, d_{-i}], x\otimes t^{-j})+\alpha(d_{-i}, [d_{i+j}, x\otimes t^{-j}]),$$
we have $jp_x(i+j)=(2i+j)p_x(j)+jp_x(-i)$.
By letting $i=-1$, there is $jp_x(j-1)=(j-2)p_x(j)+jp_x(1)$.
Because $\rho(x\otimes 1)=\alpha^\prime(d_1, x\otimes t^{-1})$, $p_x(1)=\alpha^\prime(d_1, x\otimes t^{-1})+\rho([d_1, x\otimes t^{-1}])=0$. Then
$(j-2)p_x(j)=jp_x(j-1)$. And we get $p_x(j)=\frac{j(j-1)}{2}p_x(2)$ by recursion. Hence
$$\alpha(d_i, x\otimes t^j)=\frac{i(i-1)}{2}p_x(2)\delta_{i+j, 0}.$$
Similarly, we get
$$\alpha(d_i, x\otimes t^j\xi)=\frac{i(i-1)}{2}p^\prime_x(2)\delta_{i+j, 0},$$
where $p^\prime_x(i)=\alpha(d_i, x\otimes t^{-i}\xi)$.

Let $a_x(i)=\alpha(h_i, x\otimes t^{-i})$. Since
$$\alpha(h_{i+j}, [d_{-i}, x\otimes t^{-j}])=\alpha([h_{i+j}, d_{-i}], x\otimes t^{-j})+\alpha(d_{-i}, [h_{i+j}, x\otimes t^{-j}]),$$
we have $ja_x(i+j)=(i+j)a_x(j)$. By letting $j=1$, there is $a_x(i)=ia_x(1)$. Hence
$$\alpha(h_i, x\otimes t^j)=ia_x(1)\delta_{i+j, 0}.$$
And similarly, we get
$$\alpha(Q_i, x\otimes t^j)=ib_x(1)\delta_{i+j, 0},$$
where $b_x(i)=\alpha(Q_i, x\otimes t^{-i})$.

Let $a_x^\prime(i)=\alpha(h_i, x\otimes t^{-i}\xi)$. Since
$$\alpha(h_{i+j}, [d_{-i}, x\otimes t^{-j}\xi])=\alpha([h_{i+j}, d_{-i}], x\otimes t^{-j}\xi)+\alpha(d_{-i}, [h_{i+j}, x\otimes t^{-j}\xi]),$$
we have
\begin{equation}\label{1}
ja_x^\prime(i+j)=(i+j)a_x^\prime(j)-p_x^\prime(-i).
\end{equation}
Putting $j=1$ in \eqref{1}, there is $a_x^\prime(i)=ia_x^\prime(1)-\frac{i(i-1)}{2}p_x^\prime(2)$. Putting $i=-1$ in \eqref{1},
there is $a_x^\prime(j)=ja_x^\prime(1)$ since $p_x^\prime(1)=0$. So
$p_x^\prime(2)=0$ and
$$\alpha(h_i, x\otimes t^j\xi)=ia_x^\prime(1)\delta_{i+j, 0}.$$
Similarly, we have $p_x(2)=0$ and
$$\alpha(Q_i, x\otimes t^j\xi)=ib_x^\prime(1)\delta_{i+j, 0},$$
where $b_x^\prime(i)=\alpha(Q_i, x\otimes t^{-i}\xi)$.

Let $q_x(i)=\alpha(G_i, x\otimes t^{-i})$. Since
$$\alpha(G_{i+j}, [d_{-i}, x\otimes t^{-j}])=\alpha([G_{i+j}, d_{-i}], x\otimes t^{-j})+\alpha(d_{-i}, [G_{i+j}, x\otimes t^{-j}]),$$
we have
\begin{equation}\label{2}
jq_x(i+j)=(2i+j)q_x(j).
\end{equation}
Putting $j=1$ in \eqref{2}, there is $q_x(i)=(2i-1)q_x(1)$. Putting $i=1$ in \eqref{2}, there is $q_x(j)=\frac{j(j+1)}{2}q_x(1)$. So
$q_x(1)=0$ and
$$\alpha(G_i, x\otimes t^j)=0.$$
Similarly, we get
$$\alpha(G_i, x\otimes t^j\xi)=0.$$
The Jacobi identity for the 2-cocycle $\alpha$ for the elements $h_{i+j}, Q_{-i}$ and $x\otimes t^{-j}$ gives $b_x(1)=0$.
The Jacobi identity for the 2-cocycle $\alpha$ for the elements $h_{i+j}, G_{-i}$ and $x\otimes t^{-j}$ gives $a_x^\prime(1)=0$.
The Jacobi identity for the 2-cocycle $\alpha$ for the elements $h_{i+j}, Q_{-i}$ and $x\otimes t^{-j}\xi$ gives $b_x^\prime(1)=-(-1)^{|x|}a_x(1)$.
Furthermore, since
$$\alpha(x\otimes t^{i+j}, [h_{-i}, y\otimes t^{-j}])=\alpha([x\otimes t^{i+j}, h_{-i}], y\otimes t^{-j})+\alpha(h_{-i}, [x\otimes t^{i+j}, y\otimes t^{-j}])$$
and $\fg=[\fg,\fg]$, we get $a_x(1)=0$. Therefore,
\begin{align*}
&\alpha(d_i, x\otimes t^j)=\alpha(h_i, x\otimes t^j)=\alpha(Q_i, x\otimes t^j)=\alpha(G_i, x\otimes t^j)=0,\\
&\alpha(d_i, x\otimes t^j\xi)=\alpha(h_i, x\otimes t^j\xi)=\alpha(Q_i, x\otimes t^j\xi)=\alpha(G_i, x\otimes t^j\xi)=0
\end{align*}
for any $x\in \fg$ and $i,j\in\bZ$.

As above, it is easy to get
$$\alpha(x\otimes t^i, y\otimes t^j)=\alpha(x\otimes t^i\xi, y\otimes t^j)=\alpha(x\otimes t^i\xi, y\otimes t^j\xi)=0$$
for any $x,y\in\fg$ and $i+j\neq 0$. Since
$$\alpha(x\otimes t^{i+j}, [d_{-i}, y\otimes t^{-j}])=\alpha([x\otimes t^{i+j}, d_{-i}], y\otimes t^{-j})+\alpha(d_{-i}, [x\otimes t^{i+j}, y\otimes t^{-j}]),$$
we have $j\alpha(x\otimes t^{i+j}, y\otimes t^{-(i+j)})=(i+j)\alpha(x\otimes t^j, y\otimes t^{-j})$. By letting $j=1$, there is
$\alpha(x\otimes t^i, y\otimes t^{-i})=i\alpha(x\otimes t, y\otimes t^{-1})$. So
$$\alpha(x\otimes t^i, y\otimes t^j)=i\alpha(x\otimes t, y\otimes t^{-1})\delta_{i+j,0}.$$
Similarly, we get
\begin{align*}
&\alpha(x\otimes t^i\xi, y\otimes t^j\xi)=i\alpha(x\otimes t\xi, y\otimes t^{-1}\xi)\delta_{i+j,0},\\
&\alpha(x\otimes t^i\xi, y\otimes t^j)=i\alpha(x\otimes t\xi, y\otimes t^{-1})\delta_{i+j,0}.
\end{align*}
Since
\begin{equation}\label{keyeq}\alpha(x\otimes t^{i+j}, [Q_{-i}, y\otimes t^{-j}\xi])=\alpha([x\otimes t^{i+j}, Q_{-i}], y\otimes t^{-j}\xi)+(-1)^{|x|}\alpha(Q_{-i}, [x\otimes t^{i+j}, y\otimes t^{-j}\xi]),\end{equation}
we get $(i+j)\alpha(x\otimes t, y\otimes t^{-1})=0$. Then $\alpha(x\otimes t, y\otimes t^{-1})=0$. Moreover, we get
$\alpha(x\otimes t\xi, y\otimes t^{-1}\xi)=\alpha(x\otimes t\xi, y\otimes t^{-1})=0$. Therefore,
$$\alpha(x\otimes t^i, y\otimes t^j)=\alpha(x\otimes t^i\xi, y\otimes t^j)=\alpha(x\otimes t^i\xi, y\otimes t^j\xi)=0$$
for any $x,y\in\fg$ and $i,j\in\bZ$.

Now, we know that any nontrivial 2-cocycle $\alpha$ on $\mathcal{L}$ can be induced by a nontrivial 2-cocycle on $\mathcal{W}$. According to the conclusion of \cite{Kv}, we get
$$\alpha(d_i, h_j)=-i^2\delta_{i+j, 0},\ \alpha(h_i, h_j)=2i\delta_{i+j, 0},\ \alpha(Q_i, G_j)=-i^2\delta_{i+j, 0},$$
and zero in all other cases. In addition, since $[\mathcal{L}, \mathcal{L}]=\mathcal{L}$, we obtain the universal central extension of $\mathcal{L}$, which is denoted by $\widehat{\mathcal{L}}$.

\begin{theo}\label{2-cocycle}
The super affine-Virasoro algebra $\widehat{\mathcal{L}}=\mathcal{L}\oplus \mathbb{C}C$ and the nonzero brackets of $\widehat{\mathcal{L}}$ are as follows:
\begin{align*}
&[d_i,d_j]=(j-i)d_{i+j},\;[d_i,h_j]=jh_{i+j}-i^2\delta_{i+j,0}C,\\
&[d_i,Q_j]=jQ_{i+j},\;[d_i,G_j]=(j-i)G_{i+j},\;[h_i,Q_j]=-Q_{i+j}, \;[h_i,G_j]=G_{i+j},\\
&[h_i,h_j]=2i\delta_{i+j,0}C,\;[G_i,G_j]=d_{i+j}+ih_{i+j}-i^2\delta_{i+j,0}C,\\
&[d_i,x\otimes t^j\xi]=jx\otimes t^{i+j}\xi,\ [d_i,x\otimes t^j]=jx\otimes t^{i+j},\\
&[h_i,x\otimes t^j\xi]=x\otimes t^{i+j}\xi,\ [Q_i,x\otimes t^j\xi]=(-1)^{|x|}x\otimes t^{i+j},\\
&[G_i,x\otimes t^j]=(-1)^{|x|}jx\otimes t^{i+j}\xi,\ [x\otimes t^i, y\otimes t^j]=[x,y]\otimes t^{i+j},\\
&[x\otimes t^i\xi, y\otimes t^j]=(-1)^{|y|}[x,y]\otimes t^{i+j}\xi,
\end{align*}
where $i,j\in\bZ$ and $x,y\in\fg$.
\end{theo}

\begin{rema} \label{witt} From the above we see that $\dim\,H^2(\mathcal L, \mathbb C)=1$, which is different from that of Lie algebras case (see \cite{LPX1}), is also different from that of the current Lie superalgebra $\fg\otimes \mathcal{A}$ (see \cite{NY}).
The key point is  the role of the Witt superalgebra $\mathcal W$, see \eqref{keyeq}. In fact, if we consider the Lie superalgebra $\mathfrak L=\mathfrak{W}\ltimes(\fg\otimes \mathcal{A})$ for a finite-dimensional basic classical simple Lie superalgebra $\frak g$, then from the similar proof of Theorem \ref{2-cocycle} in this paper and Proposition 3.14 in \cite{NY} we can obtain that $\dim\,H^2(\mathfrak L, \mathbb C)=2$.
\end{rema}

It is clear that $\widehat{\mathcal{L}}$ has a $\mathbb{Z}$-grading by the eigenvalues of the adjoint action of $d_0$. Then
$$\widehat{\mathcal{L}}=\bigoplus_{n\in\mathbb{Z}}\widehat{\mathcal{L}}_n=\widehat{\mathcal{L}}_+\oplus \widehat{\mathcal{L}}_0\oplus \widehat{\mathcal{L}}_-,$$
where $$\widehat{\mathcal{L}}_{\pm}=\bigoplus_{n\in\mathbb{N}}\widehat{\mathcal{L}}_{\pm n},\ \ \widehat{\mathcal{L}}_0=\fg\oplus \mathbb{C}C+\mathbb{C}d_0+\mathbb{C}h_0+\mathbb{C}Q_0+\mathbb{C}G_0.$$

Let $\fh$ be the Cartan subalgebra of Lie superalgebra $\fg$. Then $\widehat{\fh}=\fh\oplus \mathbb{C}C+\mathbb{C}d_0+\mathbb{C}h_0$ be the cartan subalgebra
of $\widehat{\mathcal{L}}$.  A highest weight module over $\widehat{\mathcal{L}}$ is characterized by
its highest weight $\Lambda\in \widehat{\fh}^*$ and highest weight vector $v_0$ such that
$(\widehat{\mathcal L}_+\oplus\fg_+)v_0=0$ and $hv_0=\Lambda(h)v_0, \forall h\in \widehat{\fh}$.

\section{The $\mathcal{A}\mathcal{L}$-modules}
Let $\cU=U(\wt{\mathcal{L}})$ and ${\cI}$ be the left ideal of $\cU $ generated by $t^i\cdot t^j-t^{i+j}$, $t^0-1$, $t^i\cdot \xi-t^i\xi$ and $\xi\cdot\xi$ for all $i, j\in\mathbb{Z}$. Then it is clear that $\cI$ is an ideal of $\cU$. Now we have the quotient algebra $\ol{\cU}=\cU/\cI=(U(\mathcal{L})U(\mathcal{A}))/\mathcal{I}$. From PBW Theorem, we may identify $\mathcal{A}$, $\mathcal{L}$ with their images in $\ol{\cU}$. Thus $\ol{\cU}=\mathcal{A}\cdot U(\mathcal{L})$. Then the category of $\mathcal{A}\mathcal{L}$-modules is equivalent to the category of $\ol{\cU}$-modules.

For $i\in \bZ\backslash \{0\}$, let $\cT$ be a subspace of $\ol{\cU}$ spanned by $\{t^{-i}\cdot d_i-d_0,\ t^{-i}\cdot Q_i-Q_0,\ t^{-i}\xi\cdot d_i-t^{-i}\cdot G_i,\
t^{-i}\xi\cdot Q_i-t^{-i}\cdot h_i,\ t^{-i}\cdot (x\otimes t^i),\ t^{-i}\xi\cdot (x\otimes t^i)-(-1)^{|x|}t^{-i}\cdot (x\otimes t^i\xi)\}$.
\begin{lemm}
\begin{enumerate}
\item $[\cT,d_0]=[\cT, Q_0]=[\cT, \mathcal{A}]=0$;
\item $\cT$ is a Lie super subalgebra of $\ol{\cU}$.
\end{enumerate}
\end{lemm}
\begin{prop}
The associative superalgebras $\ol{\cU}$ and $\cK\otimes U(\cT)$ are isomorphic.
\end{prop}
\begin{proof}
Note that $U(\cT)$ is an associative subalgebra of $\ol{\cU}$. Define the map $\phi: \cK\otimes U(\cT)\rightarrow \ol{\cU}$ by
$\phi(x\otimes y)=x\cdot y, \forall x\in \cK, y\in U(\cT)$. Then the restrictions of $\phi$ on $\cK$ and $U(\cT)$ are
well-defined homomorphisms of associative superalgebras. Note that $[\cT, d_0]=[\cT, Q_0]=[\cT, \mathcal{A}]=0$, $\phi$ is well-defined
homomorphism of associative superalgebras. From
\begin{align*}
&\phi(t^i\otimes (t^{-i}\cdot d_i-d_0)+t^id_0\otimes 1)=d_i,\ \phi(\xi Q_i\otimes 1-t^i\otimes (t^{-i}\xi \cdot Q_i-t^{-i}\cdot h_i))=h_i,\\
&\phi(t^i\otimes (t^{-i}\cdot Q_i-Q_0)+t^iQ_0\otimes 1)=Q_i, \ \phi(\xi d_i\otimes 1-t^i\otimes (t^{-i}\xi \cdot d_i-t^{-i}\cdot G_i))=G_i, \\
&\phi(t^i\otimes (t^{-i}\cdot (x\otimes t^i)))=x\otimes t^i, \\
&\phi(t^i\xi\otimes(t^{-i}\cdot (x\otimes t^i))-t^i\otimes (t^{-i}\xi\cdot (x\otimes t^i)-(-1)^{|x|}t^{-i}\cdot (x\otimes t^i\xi)))=(-1)^{|x|}x\otimes t^i\xi,
\end{align*}
we can see that $\phi$ is surjective.

By PBW theorem, we know that $\ol{\cU}$ has a basis consisting of monomials in variables
$\{d_i, h_j, Q_i, G_j, x\otimes t^j, x\otimes t^j\xi\mid i\in \bZ\setminus\{0\}, j\in \bZ, x\in \fg\}$ over $\cK$. Therefore $\ol{\cU}$
has an $\cK$-baisis consisting of monomials in the variables $\{t^{-i}\cdot d_i-d_0,\ t^{-i}\cdot Q_i-Q_0,\ t^{-j}\xi\cdot d_j-t^{-j}\cdot G_j,\
t^{-j}\xi\cdot Q_j-t^{-j}\cdot h_j,\ t^{-j}\cdot (x\otimes t^j),\ t^{-j}\xi\cdot (x\otimes t^j)-(-1)^{|x|}t^{-j}\cdot (x\otimes t^j\xi)\mid i\in \bZ\setminus\{0\}, j\in \bZ, x\in \fg\}$.
So $\phi$ is a injective and hence an isomorphism.
\end{proof}
Since the generators of $\cT$ are complex, we hope to find a simpler expression of $\cT$. Let $\fm$ be the maximal ideal of $\mathcal{A}$ generated by $t-1$ and $\xi$. Then $\fm\Delta$ is
a Lie super subalgebra of $\mathcal{W}=\mathcal{A}\Delta$ (see Section 1). Note that $\fm^{k+1}\Delta$ is spanned by the set
$$\{(t-1)^{k+1}d_i, (t-1)^{k+1}Q_i, (t-1)^kh_i, (t-1)^kG_i\mid i\in \bZ, k\in \bZ_{+}\}$$
and $\fm^k\mathcal{A}$ is spanned by the set
$$\{(t-1)^kt^i,(t-1)^kt^i\xi, (t-1)^{k-1}t^i\xi\mid i\in \bZ, k\in\bZ_{+}\}.$$
By Lemma \ref{L1}, it is easy to get the following lemma.
\begin{lemm}\label{L2}
For $k\in\bZ_+$, let $\fa_k={\fm}^{k+1}\Delta\ltimes(\fg\otimes{\fm}^k\mathcal{A})$. Then
\begin{enumerate}
\item $\fa_0$ is a Lie super subalgebra of $\mathcal{L}$;
\item $\fa_k$ is an ideal of $\fa_0$;
\item $[\fa_1,\fa_k]\subseteq\fa_{k+1}$;
\item $[\fa_0,\fa_0]\supseteq\fa_1$;
\item The ideal of $\fa_0$ generated by $\fm^k\Delta$ contains $\fa_{k}$.
\end{enumerate}
\end{lemm}
\begin{lemm}\label{L5}
(\cite{XL2}, Lemma 3.7)\ $\fm \Delta/\fm^2 \Delta\cong \gl(1,1)$.
\end{lemm}
\begin{lemm}\label{L3}
Any finite-dimensional $\fm\Delta$-module is annihilated by $\fm^k\Delta$ for large $k$.
\end{lemm}
\begin{proof}
Let $V$ be any finite-dimensional $\fm\Delta$-module. Let $\Delta'={\rm span}_\bC\{\frac{\partial}{\partial t}, \frac{\partial}{\partial\xi}\}={\rm span}_\bC
\{d_{-1}, \\
Q_0\}$ and $\mathcal{A}^+=\bC[t,\xi]$. Let $d=(t-1)d_{-1}+\xi Q_0$ and $\fm^+=\fm\cap \mathcal{A}^+$. Then $\fm^+=\oplus_{i=1}^{+\infty}\fm_i^+$
with
$$\fm_i^+=\{x\in \fm^+\mid[d, x]=ix\}=\mbox{span}\{(t-1)^i, (t-1)^{i-1}\xi\}$$
and $\fm^+\Delta'$ is a Lie super subalgebra of $\mathcal{W}=\mathcal{A}\Delta'$. Let $f(\lambda)$ be the characteristic polynomial of $d$ as an operator on $V$.
Then there exists some $p\in \bN$ such that $(f(\lambda-l),f(\lambda))=1$ for all $l\geq p$. Since
\begin{align*}
&[d, (t-1)^{l+1}d_{-1}]=l(t-1)^{l+1}d_{-1},\ \ [d, (t-1)^lh_0]=l(t-1)^lh_0,\\
&[d, (t-1)^{l+1}Q_0]=l(t-1)^{l+1}Q_0, \ \ [d, (t-1)^lG_{-1}]=l(t-1)^lG_{-1},
\end{align*}
and $\fm^+_{l+1}\Delta'=\mbox{span}\{(t-1)^{l+1}d_{-1}, (t-1)^lh_0, (t-1)^{l+1}Q_0, (t-1)^{l}G_{-1}\}$, we have $[d, d']=ld',\forall d'\in \fm^+_{l+1}\Delta'$. So $f(d-l)d'v=d'f(d)v=0$,
which implies that $d'v=0, \forall d'\in \fm^+_{l+1}\Delta',v\in V$. That is $(\fm^+)^{p+1}\Delta' V=0$.

Let $\fb$ be the ideal of $\fm\Delta'=\fm\Delta$ generated by $(\fm^+)^{p+1}\Delta'$. Then $\fb V=0$. For any
$x\in (\fm^+)^{p+1}, y\in \fm^+_1$ and $i\in\bZ$, since $xd\in \fb, yxd\in \fb$ and
$$[yt^id, xd]-[t^id, yxd]=[y,xd]t^id+y[t^id, xd]-[t^id, y]xd-y[t^id, xd]=-2t^iyxd,$$
we have $\fm^{p+2}d\subseteq \fb$. So $$zy\partial=[zd, y\partial]+(-1)^{(|y|+|\partial|)|z|}y[\partial, z]d\in \fb$$
for any $z\in \fm^{p+2}, y\in \fm^+_1$ and $\partial\in \Delta'$. Therefore, $\fm^{p+3}\Delta'\subseteq \fb$. By letting $k=p+3$, we get $\fm^k\Delta V=0$.
\end{proof}

\begin{prop}
The Lie superalgebras $\cT$ and $\fa_0$ are isomorphic.
\end{prop}
\begin{proof}
Obviously, $\fa_0={\fm}\Delta\ltimes(\fg\otimes \mathcal{A})$ is spanned by the set $\{d_i-d_0, h_i, Q_i-Q_0, G_i, x\otimes t^i, x\otimes t^i\xi\mid i\in\bZ, x\in \fg\}$.
It is easy to verify that the linear map $\varphi: \cT\rightarrow \fa_0$ defined by
\begin{align*}
&\varphi(t^{-i}\cdot d_i-d_0)=d_i-d_0,\ \ \varphi(t^{-i}\xi\cdot Q_i-t^{-i}\cdot h_i)=-h_i,\\
&\varphi(t^{-i}\cdot Q_i-Q_0)=Q_i-Q_0, \ \ \varphi(t^{-i}\xi\cdot d_i-t^{-i}\cdot G_i)=-G_i,\\
&\varphi(t^{-i}\cdot (x\otimes t^i))=x\otimes t^i, \ \ \varphi(t^{-i}\xi\cdot (x\otimes t^i)-(-1)^{|x|}t^{-i}\cdot (x\otimes t^i\xi))=-(-1)^{|x|}x\otimes t^i\xi
\end{align*}
is a Lie superalgebra isomorphism.
\end{proof}

For any $\fa_0$-module $V$, we have the $\mathcal{A}\mathcal{L}$-module $\Gamma(\lambda, V):=(\mathcal{A}(\lambda)\otimes V)^{\varphi_1}$, where
$\varphi_1: \ol{\cU}\stackrel{\phi^{-1}}{\longrightarrow}\cK\otimes U(\cT)\stackrel{1\otimes \varphi}{\longrightarrow} \cK\otimes U(\fa_0)$.
More precisely, $\Gamma(\lambda, V)=\mathcal{A}\otimes V$ with actions
\begin{align*}
&t^i\xi^r\circ(a\otimes v)=t^i\xi^ra\otimes v,\\
&d_i\circ(a\otimes v)=t^ia\otimes(d_i-d_0)\centerdot v+t^i(\lambda a+d_0(a))\otimes v,\\
&h_i\circ(a\otimes v)=t^ia\otimes h_i\centerdot v+\xi(Q_i(a))\otimes v,\\
&Q_i\circ(a\otimes v)=(-1)^{|a|}t^ia\otimes(Q_i-Q_0)\centerdot v+t^i(Q_0(a))\otimes v,\\
&G_i\circ(a\otimes v)=(-1)^{|a|}t^ia\otimes G_i\centerdot v+\xi(\lambda a+d_i(a))\otimes v,\\
&(x\otimes t^i)\circ(a\otimes v)=(-1)^{|a||x|}t^ia\otimes (x\otimes t^i)\centerdot v,\\
&(x\otimes t^i\xi)\circ(a\otimes v)=(-1)^{(|a|+1)|x|}t^i\xi a\otimes (x\otimes t^i)\centerdot v+(-1)^{|a|(|x|+1)}t^{-i}a\otimes (x\otimes t^i\xi)\centerdot v
\end{align*}
for any $a\in \mathcal{A}, v\in V$.
\begin{lemm}\label{L7}
1. For any $\lambda\in \bC$ and any simple $\fa_0$-module $V$, $\Gamma(\lambda, V)$ is a simple weight $\mathcal{A}\mathcal{L}$-module.

2. Any simple weight $\mathcal{A}\mathcal{L}$-module $M$ is isomorphic to some $\Gamma(\lambda, V)$ for some $\lambda\in \supp(M)$ and
some $\fa_0$-module $V$.
\end{lemm}
\begin{proof}
1. From Lemmas \ref{XL2-1} and \ref{XL2-2}, we know that $\mathcal{A}(\lambda)\otimes V$ is a simple $\cK\otimes U(\cT)$-module for any
$\lambda\in \bC$ and any simple $\fa_0$-module $V$. From the definition of $\Gamma(\lambda, V)$, we have the first statement.

2. Let $M$ be any simple weight $\mathcal{A}\mathcal{L}$-module with $\lambda\in \supp(M)$. Then $M^{\varphi_1^{-1}}$ is a simple $\cK\otimes U(\fa_0)$-module.
Fix a nonzero homogeneous element $v\in (M^{\varphi_1^{-1}})_\lambda$. Since $V'=\bC[\frac{\partial}{\partial\xi}]v$ is a finite-dimensional
super subspace with $\frac{\partial}{\partial\xi}$ acting nilpotently, we may find a nonzero homogeneous element $v'\in V'$ with $\cI_\lambda v'=0$.
Clearly, $\cK v'$ is isomorphic to $\mathcal{A}(\lambda)$ or $\Pi(\mathcal{A}(\lambda))$. From Lemma \ref{XL2-1}, there exists a simple $U(\fa_0)$-module $V$ such
that $M^{\varphi_1^{-1}}\cong \mathcal{A}(\lambda)\otimes V$ or $M^{\varphi_1^{-1}}\cong \Pi(\mathcal{A}(\lambda))\otimes V$. Furthermore, there is
$\Pi(\mathcal{A}(\lambda))\otimes V\cong \mathcal{A}(\lambda)\otimes \Pi(V^T)$ by Lemma \ref{XL2-2}. So this conclusion holds.
\end{proof}

Now, to classify all simple weight $\mathcal{A}\mathcal{L}$-modules, it suffices to classify all simple $\fa_0$-modules. In particular, to classify all simple
cuspidal $\mathcal{A}\mathcal{L}$-modules, it suffices to classify all simple finite-dimensional $\fa_0$-modules. Now we introduce a known conclusion that we are going to use.
\begin{lemm}\label{L4}
(\cite{Mo}, Theorem 2.1, Engel's Theorem for Lie superalgebras) Let $V$ be a finite-dimensional module for the Lie superalgebra $L=L_{\bar{0}}\oplus L_{\bar{1}}$
such that the elements of $L_{\bar{0}}$ and $L_{\bar{1}}$ respectively are nilpotent endmorphisms of $V$. Then there exists a nonzero element $v\in V$
such that $xv=0$ for all $x\in L$.
\end{lemm}
\begin{lemm}\label{L6}
1. Let $V$ be any finite-dimensional $\fa_0$-module. Then there exists $k\in \bN$ such that $\fa_kV=0$.

2. Let $V$ be any simple finite-dimensional  $\fa_0$-module. Then $\fa_1V=0$.
\end{lemm}
\begin{proof}
1. Let $V$ be any finite-dimensional $\fa_0$-module. Then $V$ is a finite-dimensional $\fm\Delta$-module. So the first statement follows from
Lemmas \ref{L2} and \ref{L3}.

2. Let $V$ be any simple finite-dimensional $\fa_0$-module. Then $V$ is a simple finite-dimensional  $\fa_0/\mbox{ann}(V)$-module, where
$\mbox{ann}(V)$ is the ideal of $\fa_0$ that annihilates $V$ and $\fa_k\subseteq \mbox{ann}(V)$ for some $k\in\bN$. So $V$ is a finite-dimensional
module for $(\fa_0)_{\bar{0}}+\mbox{ann}(V)$. By Lemma \ref{L2}, we have
$$((\fa_1)_{\bar{0}}+\mbox{ann}(V))^{k-1}\subseteq (\fa_k)_{\bar{0}}+\mbox{ann}(V)=\mbox{ann}(V),$$
which implies that $(\fa_1)_{\bar{0}}+\mbox{ann}(V)$ acts nilpotently on $V$. Since $[x, x]\in (\fa_1)_{\bar{0}}+\mbox{ann}(V)$ for any
$x\in (\fa_1)_{\bar{1}}+\mbox{ann}(V)$, every element in $(\fa_1)_{\bar{1}}+\mbox{ann}(V)$ acts nilpotently on $V$. Hence, by Lemma \ref{L4}, there is
a nonzero element $v\in V$ annihilated by $\fa_1+\mbox{ann}(V)$.

Let $V'=\{v\in V|xv=0,\forall x\in \fa_1\}$. So $V'\neq \emptyset$. For any $y\in \fa_0, x\in \fa_1$ and $v\in V'$, there is
$xyv=(-1)^{|x||y|}yxv+[x, y]v=0$, which implies that $yv\in V'$. Hence $V'=V$ by the simplicity of $V$. And therefore $\fa_1V=0$.
\end{proof}
Let $V$ be any simple finite-dimensional  $\fa_0$-module. From Lemmas \ref{L5} and \ref{L6}, we know that $V$ is a simple finite-dimensional  module for $\fa_0/\fa_1\cong \gl(1,1)\oplus \fg$. So, to classify all simple cuspidal $\mathcal{A}\mathcal{L}$-modules, it suffices to classify all simple finite-dimensional  module for $\gl(1,1)\oplus \fg$.

\section{Cuspidal modules}
In this section, we will classify all simple cuspidal modules for $\mathcal{L}$ by using the $\mathcal{A}$-cover theory.
Let $I=\fg\otimes \mathcal{A}$ and $\fri=\delta_{IM,0}\mathcal{W}+(1-\delta_{IM,0})I$. Consider $\mathcal{L}$ as the adjoint $\mathcal{L}$-module.
For an $\mathcal{L}$-module $M$, we can make the tensor product $\mathcal{L}$-module
$\fri\otimes M$ into an $\mathcal{A}\mathcal{L}$-module by defining
$$ a \cdot (\omega \otimes v)=(a\omega)\otimes v, \forall a\in \mathcal{A}, \omega\in \fri. $$
Denote $K(M)=\{ \sum\limits_{i=1}^k \omega_i\otimes v_i\in \fri\otimes M\mid\sum\limits_{i=1}^k(a\omega_i) v_i=0,\forall a\in \mathcal{A}\}$. Then it is easy to see that
$K(M)$ is an $\mathcal{A}\mathcal{L}$-submodule of $\fri\otimes M$. Then we have the $\mathcal{A}\mathcal{L}$-module $\widehat{M}=(\fri\otimes M)/K(M)$. As in \cite{BF},  we call $\widehat{M}$ as the cover of $M$ if $\fri M=M$.

Clearly,  the linear map
\begin{equation}\label{cover}\begin{split}
\pi: \,\,\widehat{M}  &\to\ \ \fri M,\\
         w\otimes y+K(M)\ & \mapsto\ \  wy,\quad \forall\ w\in \fri, y\in M
\end{split}\end{equation}
is an $\mathcal{L}$-module epimorphism.

Recall that in \cite{BF}, the authors show that every cuspidal $\mathfrak{W}$-module is annihilated by the operators $\Omega_{k,s}^{(m)}$ for $m$ large enough.

\begin{lemm}\label{Omegaoper}(\cite{BF}, Corollary 3.7)
For every $l\in\bN$ there exists $m\in\bN$ such that for all $k, s\in\bZ$ the differentiators $\Omega_{k, s}^{(m)}=\sum\limits_{i=0}^m(-1)^i\binom{m}{i}d_{k-i}d_{s+i}$ annihilate every cuspidal $\mathfrak{W}$-module with a composition series of length $l$.
\end{lemm}

For $\mathcal{L}=\mathcal{W}\ltimes(\fg\otimes \mathcal{A})$, we also want to find some operators belonging to $U(\mathcal{L})$ that can
annihilate a given cuspidal $\mathcal{L}$-module $M$. Obviously, $M$ is a cuspidal $\mathfrak{W}$-module and hence there exists $m\in\bN$ such that $\Omega_{k,p}^{(m)}M=0, \forall k,p\in\bZ$. Therefore, $[\Omega_{k,p}^{(m)},x\otimes t^j]M=0, \forall j,k,p\in\bZ, x\in \fg$. From Lemma 4.4 in {\cite{CLW}}, the authors show that $$\sum\limits_{i=0}^{m+2}(-1)^i\binom{m+2}{i}x\otimes t^{j+k+1-i}d_{p-1+i}=0.$$
Similarly, from $[\Omega_{k,p}^{(m)},x\otimes t^j\xi]M=0$ and $[d_i,x\otimes t^k\xi]=kx\otimes t^{i+k}\xi, \forall i,j,k,p\in\bZ, x\in \fg$, we have
\begin{align*}
&\sum\limits_{i=0}^{m+2}(-1)^i\binom{m+2}{i}x\otimes t^{j+k+1-i}\xi d_{p-1+i}=0.
\end{align*}
Thus, we have the following lemma.
\begin{lemm}\label{omega}
Let $M$ be a cuspidal module over $\mathcal{L}$. Then there exists $m\in\bN$ such that for all $j, p\in\bZ$ and $x\in \fg$, the operators
$\sum\limits_{i=0}^{m}(-1)^i\binom{m}{i}yd_{p+i}$ annihilate $M$, where $y\in \{x\otimes t^{j-i}, x\otimes t^{j-i}\xi \}.$
\end{lemm}

\begin{lemm}\label{cuspidalcover}
Suppose $\fg$ is finite-dimensional. Let $M$ be a cuspidal module for the Lie superalgebra $\mathcal{L}$. Then its $\mathcal{A}$-cover $\widehat{M}$ is cuspidal.
\end{lemm}
\begin{proof}
The case of $IM=0$ is proved in \cite{XL2}. Now suppose $IM\neq0$, so $\fri=I$. Since $\widehat{M}$ is an $\mathcal{A}$-module,
it is sufficient to show that one of its weight spaces is finite-dimensional, then all other weight spaces will have the same dimension. Fix a weight
$\alpha+p,\ p\in \mathbb{Z}$ and let us prove that $\widehat{M}_{\alpha+p}=\mathrm{Span}\{(x\otimes t^{p-k})\otimes M_{\alpha+k}, (x\otimes t^{p-k}\xi) \otimes M_{\alpha+k}\mid k\in\mathbb{Z}, x\in \fg\}$ is finite-dimensional. Assume that $\alpha=0$ in the case that $\alpha+\mathbb{Z}=\mathbb{Z}$,
which means that $\alpha+p\neq 0, \forall p\in \mathbb{Z}$.

We will prove by induction on $|q|$ for $q\in \mathbb{Z}$ and for all $u\in M_{\alpha+q}$,
\begin{align*}
(x\otimes t^{p-q})\otimes u, (x\otimes t^{p-q}\xi)\otimes u
&\in \sum\limits_{|k|\le\frac{m}{2}}\Big((x\otimes t^{p-k})\otimes M_{\alpha+k}+(x\otimes t^{p-k}\xi) \otimes M_{\alpha+k}\Big)+K(M).
\end{align*}
If $|q|\leq \frac{m}{2}$, the claim holds. If
$|q|>\frac{m}{2}$, we may assume $q<-\frac{m}{2}$. The proof for $q>\frac{m}{2}$ is similar. Since
$d_0$ acts on $M_{\alpha+q}$ with a nonzero scalar , we can write $u=d_0v$ for some $v\in M_{\alpha+q}$. Then
$$(x\otimes t^{p-q})\otimes d_0v=\sum\limits_{i=0}^{m}(-1)^i\binom{m}{i}(x\otimes t^{p-q-i})\otimes d_iv-\sum\limits_{i=1}^{m}(-1)^i\binom{m}{i}(x\otimes t^{p-q-i})\otimes d_iv.$$
From Lemma \ref{omega}, there exists $m\in\bN$ such that
$\sum\limits_{i=0}^{m}(-1)^i\binom{m}{i}x\otimes t^{j-i}d_{p+i}v=0$ and $\sum\limits_{i=0}^{m}(-1)^i\binom{m}{i}x\otimes t^{j-i}\xi d_{p+i}v=0$ for all $j, p\in\bZ, x\in \fg$ and
$v\in M.$ Note that $I$ has a natural module structure over the commutative superalgebra $\mathcal{A}$
$$t^i(x\otimes t^j)=x\otimes t^{i+j},\ t^i(x\otimes t^j\xi)=x\otimes t^{i+j}\xi,\ t^i\xi(x\otimes t^j)=x\otimes t^{i+j}\xi,\ t^i\xi(x\otimes t^j\xi)=0$$
for $i,j\in\bZ$ and $x\in\fg$. Hence, we have
\begin{align*}\label{1}
&\sum\limits_{i=0}^{m}(-1)^i\binom{m}{i}(x\otimes t^{j-i})\otimes d_{p+i}v,\ \sum\limits_{i=0}^{m}(-1)^i\binom{m}{i}(x\otimes t^{j-i}\xi)\otimes d_{p+i}v \in K(M)
\end{align*}
for all $j, p\in\bZ, x\in \fg$ and $v\in M$. Therefore,
$$(x\otimes t^{p-q})\otimes d_0v\in \sum\limits_{|k|\le\frac{m}{2}}\Big((x\otimes t^{p-k})\otimes M_{\alpha+k}+(x\otimes t^{p-k}\xi) \otimes M_{\alpha+k}\Big)+K(M).$$
Similarly, we have
\begin{align*}
&(x\otimes t^{p-q}\xi)\otimes d_0v\\
=&\sum\limits_{i=0}^{m}(-1)^i\binom{m}{i}(x\otimes t^{p-q-i}\xi)\otimes d_iv-\sum\limits_{i=1}^{m}(-1)^i\binom{m}{i}(x\otimes t^{p-q-i}\xi)\otimes d_iv\\
&\in \sum\limits_{|k|\le\frac{m}{2}}\Big((x\otimes t^{p-k})\otimes M_{\alpha+k}+(x\otimes t^{p-k}\xi) \otimes M_{\alpha+k}\Big)+K(M).
\end{align*}
So the lemma follows from the fact that $\dim M_{\alpha+k}<\infty$ for any fixed $k$ and $\fg$ is finite-dimensional.
\end{proof}

The following theorem gives a classification of simple cuspidal $\mathcal{L}$-modules.
\begin{theo}\label{theorem1}
Any simple cuspidal $\mathcal{L}$-module is isomorphic to a simple quotient of a tensor module $\Gamma(\lambda, V)$ for some simple finite-dimensional $\fa_0$-module $V$ and some $\lambda\in \bC$.
\end{theo}
\begin{proof}
Let $M$ be a simple cuspidal $\mathcal{L}$-module. If $M$ is a trival module of $\mathcal{L}$, then $M$ is a simple quotient of the simple cuspidal $\mathcal{A}\mathcal{L}$-module $\mathcal{A}\otimes\bC$ with $\bC$ a trivial module for $\mathcal{L}$. Now suppose $M$ is a non-trivial simple cuspidal $\mathcal{L}$-module. If $IM=0$, then $M$ is a simple cuspidal $\mathcal{W}$-module. So $M$ is a simple quotient of a simple cuspidal $\mathcal{A}\mathcal{W}$-module by Theorem 3.11 in \cite{XL2}. Since $I$ is an ideal, any $\mathcal{A}\mathcal{W}$-module is naturally an $\mathcal{A}\mathcal{L}$-module with trivial $I$ action. If $IM\neq 0$, then $IM=M$ since $M$ is simple. So there is an epimorphism $\pi:\widehat{M}\rightarrow M$. From Lemma \ref{cuspidalcover}, $\widehat{M}$ is cuspidal. Hence $\widehat{M}$ has a composition series of $\mathcal{A}\mathcal{L}$- submodules
$$0=\widehat{M}^{(1)}\subset \widehat{M}^{(2)}\subset\cdots \subset \widehat{M}^{(s)}=\widehat{M}$$ with $\widehat{M}^{(i)}/\widehat{M}^{(i-1)}$ being simple $\mathcal{A}\mathcal{L}$-modules. Let $l$ be the minimal integer such that $\pi(\widehat{M}^{(l)})\ne 0$.  Since $M$ is simple $\mathcal{L}$ module, we have $\pi(\widehat{M}^{(l)})=M$ and $\pi(\widehat{M}^{(l-1)})=0$. This gives us an epimorphism of $\mathcal{L}$-modules from $\widehat{M}^{(l)}/\widehat{M}^{(l-1)}$ to $M$. From Lemma \ref{L7}, we have $\widehat{M}^{(l)}/\widehat{M}^{(l-1)}$ is isomorphic to a tensor module  $\Gamma(\lambda,V)$ for some simple finite-dimensional $\fa_0$-module $V$ and some $\lambda\in \bC$. This completes the proof.
\end{proof}

\begin{rema}
For Lie superalgebra $\mathfrak L=\mathfrak W\ltimes(\fg\otimes\mathcal A)$, by letting $\bar{\fg}=\fg\otimes \Lambda(1)$, we get $\mathfrak L=\mathfrak W\ltimes (\bar{\fg}\otimes \bC[t^{\pm 1}])$. This shows that $\mathfrak L$ is exactly the superalgebra studied in \cite{CLW}. Therefore, the conclusions on the cuspidal modules over $\mathfrak L$, even Harish-Chandra modules over $\mathfrak L$ can be directly obtained from \cite{CLW}.
\end{rema}

\section{Simple Harish-Chandra modules over super affine-Virasoro algebras}
In this section, we will classify all simple Harish-Chandra modules over super affine-Virasoro algebras $\widehat{\mathcal{L}}$.
Let $\{x_s\mid s=1,2,\cdots, l\}$ be a basis of Lie superalgebra $\fg=\fg_{\bar{0}}\oplus \fg_{\bar{1}}$. Then $\mbox{dim}\,\widehat{\mathcal{L}}_n=4+2l$ for $n\in\mathbb{Z}$ and $n\neq 0$.
Let $M$ be a simple Harish-Chandra module over $\widehat{\mathcal{L}}$. By Schur's Lemma, we may assume that the central element $C$ acts
on $M$ by scalar $c$.
\begin{lemm}\label{6.1}
Suppose that $M=\bigoplus_{\lambda\in \mathbb{C}}M_\lambda$ is a simple cuspidal module over $\widehat{\mathcal{L}}$. Then the action of central element $C$
on $M$ is trival.
\end{lemm}
\begin{proof}
Let $\bar{d}_i=d_i+\frac{1}{2}if_i$ for $i\in \mathbb{Z}$. Then
$$[\bar{d}_i, \bar{d}_j]=(j-i)\bar{d}_{i+j}+\frac{1}{2}j^3\delta_{i+j,0}C.$$
So $\fD={\rm span}\{\bar{d}_i, C\mid i\in \mathbb{Z}\}$, which is isomorphic to the Virasoro algebra,  is a subalgebra of $\widehat{\mathcal{L}}$.  Note that $M$ is a cuspidal
module over the Virasoro algebra $\fD$. By \cite{KS}, we have $c=0$.
\end{proof}
From Lemma \ref{6.1}, we know that the category of simple cuspidal $\widehat{\mathcal{L}}$-modules is naturally equivalent to the category of
simple cuspidal $\mathcal{L}$-modules. Thus, it remains to classify all simple Harish-Chandra modules over $\widehat{\mathcal{L}}$ which is not cuspidal.
The following result is well known.
\begin{lemm}\label{6.2}
Let $V$ be a weight module with finite-dimensional weight spaces for the Witt algebra $\mathfrak{W}$ with supp$(V)\subseteq \lambda+\mathbb{Z}$.
If for any $v\in V$, there exists $N(v)\in \mathbb{N}$ such that $d_iv=0, \forall i\geq N(v)$, then supp$(V)$ is upper bounded.
\end{lemm}
\begin{lemm}\label{6.3}
Suppose $M$ is a simple Harish-Chandra module over $\widehat{\mathcal{L}}$ which is not cuspidal, then $M$ is a highest (or lowest) weight module.
\end{lemm}
\begin{proof}
For a fixed $\lambda\in\mbox{supp}(M)$, there is a $k\in\mathbb{Z}$ such that $\mbox{dim}M_{\lambda-k}>(4+2l)\mbox{dim}M_\lambda+4\mbox{dim}M_{\lambda+1}$
since $M$ is not cuspidal. Without loss of generality, we may assume that $k\in \mathbb{N}$. Then there exists a nonzero element $\omega\in M_{\lambda-k}$
such that
$$d_k\omega=d_{k+1}\omega=h_k\omega=h_{k+1}\omega=Q_k\omega=Q_{k+1}\omega=G_k\omega=G_{k+1}\omega=0$$
and
$$x_s\otimes t^k\omega=x_s\otimes t^k\xi\omega=0,$$
where $s\in\{1,2,\cdots,l\}$. Therefore, we get
$d_p\omega=h_p\omega=Q_p\omega=G_p\omega=x_s\otimes t^p\omega=x_s\otimes t^p\xi\omega=0$ for all $p\geq k^2$
since $[\widehat{\mathcal{L}}_i, \widehat{\mathcal{L}}_j]=\widehat{\mathcal{L}}_{i+j}$ for $j\neq 0$.

It is easy to see that $M'=\{v\in M\mid\mbox{dim}\ \widehat{\mathcal{L}}_+v<\infty\}$ is a nonzero submodule of $M$. So $M=M'$ by the simplicity of $M$.
Since $M$ is also the $d_0$-weight module over the Witt algebra $\mathfrak{W}$, we have $\mbox{supp}(M)$ is upper bounded by Lemma \ref{6.2}, that is $M$ is a highest weight module.
\end{proof}
Combining with Lemma \ref{L7}, Theorem \ref{theorem1} and Lemma \ref{6.3}, we have the following result.
\begin{theo}\label{Theorem 3}
Let $M$ be a simple Harish-Chandra module over $\widehat{\mathcal{L}}$. Then $M$ is a highest weight module, a lowest weight module, or
a simple quotient of a tensor module $\Gamma(\lambda, V)$ for some simple finite-dimensional $\fa_0$-module $V$ and some $\lambda\in \bC$.
\end{theo}

Let $\fg$ be a finite-dimensional basic classical simple Lie superalgebra. From Remark 3.2, we see that $\dim\, H^2(\mathfrak L, \bC)=2$.
By the similar proof as Lemma 3.2 in \cite{LPX1} and the conclusions in \cite{CLW}, we can get a similar result for the Lie superalgebra $\widehat{\mathfrak{L}}$.

\begin{rema}
Let $\fg$ be a finite-dimensional basic classical simple Lie superalgebra, and $M$ be a simple Harish-Chandra module over $\widehat{\mathfrak{L}}$. Then $M$ is a highest weight module, a lowest weight module, or a simple quotient of a tensor module $\Gamma(\lambda, V)$ for some simple finite-dimensional $\frak b_0$-module $V$ and some $\lambda\in \bC$,
where $\frak b_0=(t-1)\mathfrak W\ltimes (\bar{\fg}\otimes \bC[t^{\pm 1}])$.
\end{rema}
\noindent {\bf Acknowledgement} The authors would like to thank the professor R. L\"{u} for his help in preparation of this paper. This work was supported by
NSF of China (Grant 12101082, 12071405, 11971315, 11871429, 11871052),  NSF for Youths of Jiangsu Province (Grant BK20201051) and
Jiangsu Provincial Double-Innovation Doctor Program (Grant JSSCBS20210742).

\end{document}